\documentclass[10pt]{article}
\usepackage{graphicx}
\usepackage{amsmath}
\usepackage{amssymb}
\usepackage{amsthm}
\usepackage{epsfig}
\usepackage{mathrsfs}
\usepackage{multirow}
\usepackage[usenames,dvipsnames]{color}
\usepackage{setspace}
\usepackage{algorithm, algorithmic} 
\usepackage{multicol}
\usepackage{subcaption}
\usepackage[margin=2.54cm]{geometry}
\usepackage{lipsum}
\usepackage{bigints}
\usepackage{framed}
\usepackage{colortbl}
\usepackage{calligra}
\usepackage{mathrsfs}
\usepackage{enumerate}
\usepackage{hyperref}
\usepackage{enumitem}
\usepackage[square, numbers, sort]{natbib}
\usepackage{cleveref}
\usepackage{xcolor}
\usepackage{mathtools}

\usepackage{appendix}

\DeclareMathAlphabet{\mathcalligra}{T1}{calligra}{m}{n}

\newtheorem{theorem}{Theorem}

\newtheorem{corollary}{Corollary}

\newtheorem{proposition}{Proposition}

\newcommand{\bm}{\boldsymbol}
\newcommand{\bs}{\boldsymbol}

\def\RR{ {\mathbb{R}}}

\newcommand{\mc}[1]{\mathcal{#1}}
\newcommand{\mb}[1]{\mathbb{#1}}

\newcommand{\xit}{\tilde{\bm{\xi}}}
\newcommand{\gammat}{\tilde{\bm{\gamma}}}

\newcommand{\bxi}{\bm \xi}

\newcommand{\EE}{{\mathbb{E}}}
\newcommand{\EErho}{{\mathbb{E}_{\bm\gamma}}}
\newcommand{\EErhohat}{{\hat{\mathbb{E}}_{\bm\gamma}}}

\title{
Generalization Bounds for Contextual Stochastic Optimization using Kernel Regression}
\author{
Yijie Wang\thanks{Graduate Program in Operations Research and Industrial Engineering, 
The University of Texas at Austin, Austin, TX 78712-1591, USA. Email: {\tt yijie-wang@utexas.edu}.}
\ \ \ \ \ \ \
Grani A. Hanasusanto\thanks{Department of Industrial and Enterprise Systems Engineering, University of Illinois Urbana-Champaign, Urbana, IL 61801, USA. Email: {\tt gah@illinois.edu}.}
\ \ \ \ \ \ \
Chin Pang Ho\thanks{School of Data Science, City University of Hong Kong, Hong Kong. Email: {\tt clint.ho@cityu.edu.hk}.}%
}
\date{}

\begin{document}

\maketitle

\begin{abstract}
In this paper, we consider contextual stochastic optimization using Nadaraya-Watson kernel regression, which is one of the most common approaches in nonparametric regression. Recent studies have explored the asymptotic convergence behavior of using Nadaraya-Watson kernel regression in contextual stochastic optimization; however, the performance guarantee under finite samples remains an open question. This paper derives a finite-sample generalization bound of the Nadaraya-Watson estimator with a spherical kernel under a generic loss function. Based on the generalization bound, we further establish a suboptimality bound for the solution of the Nadaraya-Watson approximation problem relative to the optimal solution. Finally, we derive the optimal kernel bandwidth and provide a sample complexity analysis of the Nadaraya-Watson approximation problem.
\end{abstract}

\section{Introduction}
In the presence of uncertainty, decision makers can leverage contextual information such as weather conditions and macroeconomic indicators to refine the decision-making process~\cite{sadana2024survey}. In this paper, we consider a general stochastic optimization framework that incorporates contextual information into the decision-making process 
\begin{equation}\label{eq:conditional_expectation_problem}
\min_{\bm x\in\mathcal X}\;\left\{ \EErho[\ell(\bm x,\xit)]\coloneqq\EE[\ell(\bm x,\xit)\,|\,\gammat=\bm\gamma]\right\},
\end{equation}
where the vector $\bm x\in\RR^d$ represents the decision variables, $\xit \in \RR^q$ denotes the uncertain parameters, and $\gammat \in \RR^p$ denotes contextual covariates. While the uncertain parameters $\xit$ and contextual covariates $\gammat$ are both random, the decision-maker can observe the realization of $\gammat$ before the decision-making stage. Hence, to exploit this side information, he solves a stochastic optimization problem with the conditional expectation of the random cost $\ell(\bm x,\xit)$ as an objective function, given the realization~$\bm\gamma$.

Unfortunately, problem~\eqref{eq:conditional_expectation_problem} is challenging to solve since the joint distribution of $(\xit, \gammat)$ is usually unknown, and only a set of historical observations $\{(\bm\xi^i,\bm\gamma^i)\}_{i=1}^n$ is available. A common approach to approximate~\eqref{eq:conditional_expectation_problem} is using Nadaraya-Watson kernel regression~\citep{nadaraya:64,watson:64}, which estimates the conditional expectation term in a data-driven manner by
\begin{equation}\label{eq:nadaraya-watson}
\displaystyle\hat\EE_{\bm\gamma}[\ell(\bm x,\xit)]=\frac{\sum_{i=1}^n\mathcal K\left(\frac{\bm\gamma-\bm\gamma^i}{h}\right)\ell(\bm x,\bm\xi^i)}{\sum_{i=1}^n\mathcal K\left(\frac{\bm\gamma-\bm\gamma^i}{h}\right)},
\end{equation}
where $\mathcal K$ is a prescribed kernel function and $h > 0$ is the bandwidth parameter. In this paper, we consider the spherical kernel~\cite{tibshirani2013nonparametric} 
\begin{equation}\label{eq:naive_kernel}
\mathcal K(\bm \theta) = \mb I(\lVert \bm \theta \rVert \leq h).
\end{equation}
Using the estimator~\eqref{eq:nadaraya-watson}, we arrive at the following Nadaraya-Watson approximation problem~\cite{Hannah:10} to the contextual stochastic optimization problem~\eqref{eq:conditional_expectation_problem}:
\begin{equation}
\tag{$\mathcal {NW}$}
\label{eq:nadaraya-watson_problem}
\min_{\bm x\in\mathcal X}\;\EErhohat[\ell(\bm x,\xit)]. 
\end{equation}

\subsection{Related Work} 
The statistical properties and convergence behavior of Nadaraya-Watson estimations have been widely studied in recent years. In the nonparametric statistics literature, existing results on the Nadaraya-Watson estimator (when the decision is fixed) primarily focus on analyzing the mean square errors and the probabilistic convergence, but they lack finite-sample analysis~\cite{belkin2018overfitting,belkin2019does,gyorfi2006distribution,tibshirani2013nonparametric}.

In the contextual stochastic optimization settings, Bertsimas and Kallus~\cite{bertsimas2014predictive} show that the Nadaraya-Watson estimator converges to the true conditional expectation as the sample size grows, meaning that the approximation problem is asymptotically consistent. They further derive a performance guarantee by assuming $\tilde{\bm{\xi}}$ or $\bm{x}$ to be linear in $\tilde{\bm{\gamma}}$. Bertsimas and Van Parys~\cite{bertsimas2017bootstrap} derive an alternative bound for the bootstrap data that is obtained through resampling from the empirical distribution. For the specific single-item newsvendor problem,  Ban and Rudin~\cite{ban2018big} employ the Nadaraya-Watson approximation with linear decision rule and derive a finite-sample performance guarantee for the solution. Unfortunately, the bound relies on the decision rule approximation and holds only for the specific one-dimensional newsvendor problem. 
For generic problem settings and loss functions,  Srivastava et al.~\cite{srivastava2021data} establish a generalization bound by leveraging techniques from large and moderate deviations theory. However, the derived generalization bound is asymptotic. Hence, the finite-sample behavior of the Nadaraya-Watson approximation problem remains an open question. 

Finally, we remark that finite-sample bounds have been established for the \emph{sample-average approximation} scheme in traditional \emph{un}conditional stochastic optimization \cite{kleywegt2002sample,SN2005}.  A recent work~\cite{hu2020sample} derives generalization bounds in a non-data-driven setting where conditional samples of $\xit$  for any given $\bm\gamma$ can be obtained. Our present work is motivated by the lack of results in \emph{data-driven} contextual stochastic optimization settings, where we utilize kernel regression as a ``conditional sample-average approximation." 

\subsection{Contribution}
We summarize below the main contributions of the paper: 
\begin{enumerate}
\item We derive a finite-sample generalization bound for the Nadaraya-Watson estimator with a spherical kernel. To the best of our knowledge, our paper is the first to establish a finite-sample bound for the Nadaraya-Watson estimator under a \emph{generic} loss function. 
\item Based on the finite-sample generalization bound, we further establish a suboptimality bound for the solution of the approximation problem~\eqref{eq:nadaraya-watson_problem} relative to the optimal solution of the true problem~\eqref{eq:conditional_expectation_problem}. As part of the result, we derive the optimal scaling rate for the kernel bandwidth $h$ and analyze the sample complexity of the Nadaraya-Wation approximation problem, providing theoretical guidance for applying kernel regression in practical contextual stochastic optimization problems. 
\end{enumerate}

\noindent \textbf{Notation and terminology} We use bold letters for vectors, while scalars are printed in regular font. Random variables are designated by tilde signs (e.g., $\xit$), while their realizations are represented by the same symbols without tildes (e.g., $\bm\xi$). For any $n\in\mathbb N$, we define $[n]$ as the  index set $\{1,\ldots,n\}$. Without further specification, $\|\cdot\|$ denotes the Euclidean norm. The Gamma function is denoted by~$\Gamma$, where $\Gamma(z) = \int_0^\infty t^{z-1} e^{-t} dt$ for any complex number $z$ with a positive real part. For asymptotic analysis, we use $\tilde{O}$ notation to denote the $O$ notation that suppresses multiplicative terms with logarithmic dependence on $n$.

\section{Finite Sample Generalization Bounds}\label{sec:generalization_bound}

In this section, we first derive generalization bounds on the approximation \eqref{eq:nadaraya-watson_problem} for a fixed decision~$\bm x$. We assume the following mild regularity conditions:
\begin{enumerate}[label=(\textbf{A\arabic*})]
\item  The feasible set $\mathcal X$ is bounded.   \label{as1}  
\item The marginal density function $f(\bs \gamma)$ exists and is uniformly bounded from below by $\underline{f}$ such that $ 0< \underline{f} \leq f(\bs \gamma)$ for all $\bs\gamma$ in its support.  \label{as2}
\item The loss function $\ell(\bm x,\bm\xi)$ is $L_x$-Lipschitz continuous in $\bm x$ takes value in the interval $[0,1]$. In addition, the conditional expectation $\mb E[\ell(\bs x,\tilde{\bxi})\lvert \tilde{\bs \gamma}=\bs \gamma]$ is $L_\gamma$-Lipschitz continuous in $\bs \gamma$.  
\label{as3} 
\end{enumerate}
The assumption about the feasible set in \ref{as1} is typical in the literature~\cite{hu2020sample,SN2005,sadana2024survey}.
The assumption about the marginal density function in \ref{as2} is merely to simplify the analysis, as it allows us to analytically bound the probability of $\tilde{\bm \gamma} \in \mathcal{B}_r$ with a simple function of $r$, for any 2-norm ball $\mathcal{B}_r$ with radius $r$.
The Lipschitz continuity assumptions in \ref{as3} are standard mild regularity conditions to establish performance guarantees of kernel regression schemes~\cite{tibshirani2013nonparametric}. The assumption that the loss function takes values in $[0,1]$ follows immediately through scaling and translation if the loss function is bounded.
This boundedness assumption on the loss function is usual in practical contextual stochastic optimization problems, especially when the feasible set of the decision variable $\mathcal X$ is bounded.
It is important to note that the convexity of the loss function is not required for our analysis.

Our main result establishes the suboptimality bound for the solution $\hat{\bm x}$ to the Nadaraya-Watson approximation problem  \eqref{eq:nadaraya-watson_problem} relative to the optimal solution $\bm x^\star$ of the true problem \eqref{eq:conditional_expectation_problem}. 

\begin{theorem}
\label{thm:suboptimality}
    Let  $D=\sup_{\bm x,\bm x'\in\mathcal X}\|\bm x-\bm x'\|$ be the diameter of $\mathcal X$.   Under the bandwidth scaling     
    \begin{equation*}
    h(n) = \left( \frac{2L_\gamma^2nc\underline{f}}{p^2\log\left(\frac{2\mathcal O(1)(D/\tau)^d}{\delta}\right)}\right)^{\frac{-1}{p+2}}
    \end{equation*}
 the following suboptimality bound for the solution $\hat{\bm x}$ of \eqref{eq:nadaraya-watson_problem} holds with probability at least $1-\delta$:
\begin{equation*}
\mb{E}_{\bm \gamma}[\ell(\hat{\bs x}, \tilde{\bs\xi})]  \leq {\mb{E}}_{\bm \gamma}[\ell(\bs x^\star, \tilde{\bs\xi})] +  2L_\gamma^{\frac{p}{p+2}} \frac{p+2}{(4p^p)^{\frac{1}{p+2}}}  \left(\frac{2\log\left(\frac{2\mathcal O(1)(D/\tau)^d}{\delta}\right)}{nc\underline{f}} \right)^{\frac{1}{p+2}}+4L_x\tau.
\end{equation*}
Here, $\tau>0$ is a free parameter that can be adjusted according to the desired approximation quality. 
\end{theorem}

 Before providing the proof of the above theorem, we first introduce its practical insights into the sample complexity of the Nadaraya-Watson problem.
\begin{corollary}\label{coro:sample_complex}
To obtain a suboptimality gap of $\epsilon$ with probability at least $1-\delta$, it is sufficient that
\begin{equation*}
n\geq  \frac{2^{2p+3}L_\gamma^p(p+2)^{p+2}\log\left(\frac{2\mathcal O(1)(8DL_x/\epsilon)^d}{\delta}\right)}{p^p\epsilon^{p+2}c\underline{f}}.
\end{equation*}
\end{corollary}
\begin{proof}[Proof of Corollary~\ref{coro:sample_complex}]
To obtain a suboptimality gap of $\epsilon$, we let $4L_x\tau=\epsilon/2$ by setting $\tau = \frac{\epsilon}{8L_x}$, and we impose \begin{equation*}
 \frac{\epsilon}{2}\geq  2L_\gamma^{\frac{p}{p+2}} \frac{p+2}{(4p^p)^{\frac{1}{p+2}}}   \left(\frac{2\log\left(\frac{2\mathcal O(1)(D/\tau)^d}{\delta}\right)}{nc\underline{f}} \right)^{\frac{1}{p+2}}.
\end{equation*}
Solving for the lower bound on $n$ yields a sufficient number of samples. This completes the proof.  
\end{proof}
The sample complexity of the Nadaraya-Watson approximation is  $\tilde{\mathcal O}(1/\epsilon^{p+2})$, which grows exponentially in the dimension of the covariates. This suggests that our purely data-driven setting, where we do not impose stringent assumptions on the joint distribution of $\gammat$ and $\xit$, suffers from the \emph{curse of dimensionality}.   A more favorable sample complexity is achievable by using the traditional sample-average approximation in a non-data-driven setting, where conditional samples of $\xit$ can be collected for any given $\bm\gamma$  \cite{hu2020sample}. Alternatively, if $\xit$ can be decomposed into a sum of linear regression function in $\bm\gamma$ and an independent random noise, then an improved sample complexity can be obtained using the \emph{residuals-based sample-average approximation} \cite{kannan2020data}.

The proof of Theorem \ref{thm:suboptimality} relies on the following result that establishes a high confidence bound on the errors of the Nadaraya-Watson estimate for any fixed $\bm x$. To our knowledge, the finite-sample generalization bound for kernel regression has not been derived in the literature on nonparametric statistics. Given that the Nadaraya-Watson estimator is biased, the proof of the proposition first decomposes the error into bias and mean-absolute deviation (MAD) terms, and then derives bounds for each term separately.

\begin{proposition}[Finite-Sample Generalization Bound]\label{thm:navie_finite_bd}
Consider the following Nadaraya-Watson estimator
\begin{equation*}
\hat{\mb E}_{\bm{\gamma}}[\ell(\bs x,\tilde{\bs \xi})] = \left\{
\begin{array}{ll}
\displaystyle \sum_{i=1}^{n}\ell(\bs x, \bs \xi^i) \frac{\mb I(\lVert \bs \gamma -\bs \gamma^i \rVert \leq h)}{\sum_{j=1}^n \mb I(\lVert \bs \gamma -\bs \gamma^j \rVert \leq h) }&\displaystyle \textup{if} \; \sum_{j=1}^n \mb I(\lVert \bs \gamma -\bs \gamma^j \rVert \leq h)>0 \\
0& \textup{otherwise.}
\end{array}\right.
\end{equation*}
For any fixed $\bs x \in \mc X$, $\epsilon \in [0,1]$ and $n \in \mb Z_{++}$, we have
\begin{equation}\label{eq:finite_sample_bound}
\left|\mb E_{\bm \gamma}[\ell(\bs x,\tilde{\bs \xi})] - \hat{\mb E}_{\bm \gamma}[\ell(\bs x,\tilde{\bs \xi})]\right| \leq L_\gamma h + \epsilon
\end{equation}
with probability at least $1-2 \exp\left(-nc\underline{f}h^p \epsilon^2/2\right)$, where $c=\frac{\pi^{p/2}}{\Gamma(p/2 + 1)}$  is a constant that depends only on $p$. 
\end{proposition}

\begin{proof}[Proof of \Cref{thm:navie_finite_bd}]
We first define
\begin{equation*}
\tilde{m}(\bs \gamma) = 
\left\{\begin{array}{ll}
\displaystyle \sum_{i=1}^{n} \mb E[\ell(\bs x, \tilde{\bs\xi})\lvert \tilde{\bs\gamma}=\bs \gamma^i] \frac{\mb I(\lVert \bs \gamma -\bs \gamma^i \rVert \leq h)}{\sum_{j=1}^n \mb I(\lVert \bs \gamma -\bs \gamma^j \rVert \leq h) } &\displaystyle \textup{if} \; \sum_{j=1}^n \mb I(\lVert \bs \gamma -\bs \gamma^j \rVert \leq h)>0 , \\
0& \textup{otherwise.}
\end{array}\right.
\end{equation*}
Then, by triangle inequality, we have 
\begin{equation}
\label{bias-MAD}
\begin{aligned}
\left \lvert \mb{E}_{\bm \gamma}[\ell(\bs x, \tilde{\bs\xi})] - \hat{\mb{E}}_{\bm \gamma}[\ell(\bs x, \tilde{\bs\xi})] \right\rvert 
\leq \underbrace{\lvert \mb{E}_{\bm \gamma}[\ell(\bs x, \tilde{\bs\xi})] -\tilde{m}(\bs \gamma) \rvert}_{(a)} + \underbrace{\lvert   \tilde{m}(\bs \gamma)-\hat{\mb{E}}_{\bm \gamma}[\ell(\bs x, \tilde{\bs\xi})] \rvert}_{(b)}.
\end{aligned}
\end{equation}
Conditioning on the covariate samples $\{ \bs \gamma^i\}_{i=1}^n$, we observe that 
\[\tilde{m}(\bs \gamma) = \mb{E}\left[\hat{\mb E}_{\bm{\gamma}}[\ell(\bs x,\tilde{\bs \xi})]\Big\lvert\{ \bs \gamma^i\}_{i=1}^n\right].\] 
Thus, the upper bound in \eqref{bias-MAD} can be interpreted as a (conditional) bias-MAD decomposition of the error, where (a) corresponds to the conditional bias of the Nadaraya-Watson estimator, while (b) corresponds to its absolute deviation from the conditional mean. In what follows, we derive an upper bound of each of the terms (a) and (b) separately. We first consider the term~(a) and denote $\mb I(\mathcal E)$ to be the indicator function for event $\displaystyle \mathcal E :=  {\displaystyle \cap }_{i=1}^n \{ \lVert \bs \gamma - \bs \gamma^i \lVert > h \} $. We have
\begin{equation*}
\displaystyle
\begin{aligned}
\lvert  \mb{E}_{\bm \gamma}[\ell(\bs x, \tilde{\bs\xi})] -\tilde{m}(\bs \gamma) \rvert 
&=  \bigg\lvert \left( \mb{E}_{\bm \gamma}[\ell(\bs x, \tilde{\bs\xi})]  - \sum_{i=1}^{n} \mb{E}_{\bm \gamma^i}[\ell(\bs x, \tilde{\bs\xi})]  \frac{\mb I(\lVert \bs \gamma -\bs \gamma^i \rVert \leq h)}{\sum_{j=1}^n \mb I(\lVert \bs \gamma -\bs \gamma^j \rVert \leq h) }\right) \mb{I}(\mathcal E^c) + \mb{E}_{\bm \gamma}[\ell(\bs x, \tilde{\bs\xi})] \mb I(\mathcal E)\bigg \rvert \\
&\leq \bigg\lvert \mb{E}_{\bm \gamma}[\ell(\bs x, \tilde{\bs\xi})] - \sum_{i=1}^{n} \mb{E}_{\bm \gamma^i}[\ell(\bs x, \tilde{\bs\xi})] \frac{\mb I(\lVert \bs \gamma -\bs \gamma^i \rVert \leq h)}{\sum_{j=1}^n \mb I(\lVert \bs \gamma -\bs \gamma^j \rVert \leq h) } \bigg \rvert + \mb I(\mathcal E) \\
&= \bigg\lvert  \sum_{i=1}^{n} \left( \mb{E}_{\bm \gamma}[\ell(\bs x, \tilde{\bs\xi})] - \mb{E}_{\bm \gamma^i}[\ell(\bs x, \tilde{\bs\xi})]\right) \frac{\mb I(\lVert \bs \gamma -\bs \gamma^i \rVert \leq h)}{\sum_{j=1}^n \mb I(\lVert \bs \gamma -\bs \gamma^j \rVert \leq h) } \bigg \rvert + \mb I(\mathcal E) \\
&\leq  \sum_{i=1}^{n} \left( L_\gamma \Vert \bm{\gamma} - \bm{\gamma}^i \Vert \cdot \bigg\lvert  \frac{\mb I(\lVert \bs \gamma -\bs \gamma^i \rVert \leq h)}{\sum_{j=1}^n \mb I(\lVert \bs \gamma -\bs \gamma^j \rVert \leq h) } \bigg \rvert \right) + \mb I(\mathcal E) \leq  L_\gamma h +  \mb I(\mathcal E),
\end{aligned}
\end{equation*}
where the first inequality follows from the fact that $\mb{E}_{\bm \gamma}[\ell(\bs x, \tilde{\bs\xi})] \leq 1$. In addition, we further have 
\begin{equation}\label{eq:term_a_bound}
\mb P (\tilde{\mathcal E}) = \mb P({\displaystyle \cap }_{i=1}^n \{ \lVert \bs \gamma - \tilde{\bs \gamma}^i \lVert > h \}) = ( 1- p_h)^n ,
\end{equation}
where $p_h = \mb P (\lVert \bs \gamma - \tilde{\bs \gamma}^i \lVert \leq h)$.
Hence, we have
\begin{equation*}
\displaystyle
\begin{aligned}
\lvert  \mb{E}_{\bm \gamma}[\ell(\bs x, \tilde{\bs\xi})] -\tilde{m}(\bs \gamma) \rvert \leq L_\gamma h
\end{aligned}
\end{equation*}
with probability at least $1- ( 1- p_h)^n$.

Next, we obtain an upper bound on the term (b):
\begin{equation*}
\begin{aligned}
\displaystyle
\lvert \tilde{m}(\bs \gamma)- \hat{\mb{E}}_{\bm \gamma}[\ell(\bs x, \tilde{\bs\xi})] \rvert =& \lvert \tilde{m}(\bs \gamma)- \hat{\mb{E}}_{\bm \gamma}[\ell(\bs x, \tilde{\bs\xi})]  \rvert \cdot \mb I  (\mathcal E^c)  + \lvert \tilde{m}(\bs \gamma)- \hat{\mb{E}}_{\bm \gamma}[\ell(\bs x, \tilde{\bs\xi})]  \rvert \cdot \mb I  (\mathcal E)\\
=&\left\lvert \left( \sum_{i=1}^{n}\left(\ell(\bs x,\bs \xi^i) - \mb E_{\bm\gamma^i}[\ell(\bs x, \tilde{\bs\xi} )]\right) \frac{\mb I(\lVert \bs \gamma -\bs \gamma^i \rVert \leq h)}{\sum_{j=1}^n \mb I(\lVert \bs \gamma -\bs \gamma^j \rVert \leq h) } \right) \cdot \mb I  (\mathcal E^c)\right\rvert + \left|(0 - 0)\cdot \mb I  (\mathcal E)\right|\\
 = &  \left\lvert \left( \sum_{i=1}^{n}\left(\ell(\bs x,\bs \xi^i) - \mb E_{\bm\gamma^i}[\ell(\bs x, \tilde{\bs\xi} )]\right) \frac{\mb I(\lVert \bs \gamma -\bs \gamma^i \rVert \leq h)}{\sum_{j=1}^n \mb I(\lVert \bs \gamma -\bs \gamma^j \rVert \leq h) } \right) \cdot \mb I  (\mathcal E^c)\right\rvert
\end{aligned}
\end{equation*}
where the second equality holds because $\tilde{m}(\bs \gamma)$ and $ \hat{\mb{E}}_{\bm \gamma}[\ell(\bs x, \tilde{\bs\xi})]$ are zero conditioned on the event $\mathcal E$.

From above, by the application of Hoeffding's inequality, we get 
\begin{equation*}
\begin{aligned}
&\mb P \left(\lvert\tilde{m}(\bs \gamma)- \hat{\mb{E}}_{\bm \gamma}[\ell(\bs x, \tilde{\bs\xi})] \rvert >\epsilon \Big\lvert \{\bs \gamma^i\}_{i=1}^n, \mathcal E^c\right)\\ 
&= \mb P\left(\left|\sum_{i: \lVert \bs \gamma -\bs \gamma^i \rVert\leq h}\left(\ell(\bs x, \tilde{\bs \xi}^i) - \mb E_{\bm\gamma^i}[\ell(\bs x, \tilde{\bs\xi} )] \right)\right| > \epsilon \cdot \sum_{j=1}^n \mb I(\lVert \bs \gamma-\bs \gamma^j \rVert \leq h)\bigg\lvert \{ \bs \gamma^i\}_{i=1}^n, \mathcal E^c\right) \\
&\leq 2 \cdot \exp\left(-2\epsilon^2 \cdot\frac{ [\sum_{j=1}^n \mb I(\lVert \bs \gamma - \bs \gamma^j\rVert \leq h)]^2}{\sum_{j=1}^n \mb I(\lVert \bs \gamma - \bs \gamma^j\rVert \leq h)} \right)\\
&= 2 \cdot \exp\left(-2 \epsilon^2  \cdot \sum_{j=1}^n \mb I(\lVert \bs \gamma - \bs \gamma^j\rVert \leq h)\right). 
\end{aligned}
\end{equation*}
Taking expectations on both sides yields
\begin{equation}\label{eq:expectation_both_sides}
\begin{aligned}
\mb P \left(\lvert\tilde{m}(\bs \gamma)- \hat{\mb{E}}_{\bm \gamma}[\ell(\bs x, \tilde{\bs\xi})] \rvert >\epsilon\right) 
&\leq 2\cdot \mb E\left[\exp\left(-2 \epsilon^2 \cdot \sum_{j=1}^n \mb I(\lVert \bs \gamma - \tilde{\bs \gamma}^j\rVert \leq h) \right)\mb I(\mathcal E^c) \right]\\
&= 2 \cdot \mb E\left[\exp\left(-2\epsilon^2 \cdot \tilde{Z}\right)\mb I(\tilde{Z}>0)\right],
\end{aligned}
\end{equation}
where $\tilde{Z} \sim \text{Bin}(n, p_h)$ is a binomial random variable with parameters $n$ and $p_h$. The last equality holds because $\sum_{j=1}^n \mb I(\lVert \bs \gamma - \tilde{\bs \gamma}^j\rVert \leq h )$ has the same distribution as~$\tilde{Z}$.  
Hence, we can compute the last expression in~\eqref{eq:expectation_both_sides} as  
\begin{equation}\label{eq:exact_expectation}
\begin{aligned}
 2 \cdot \mb E\left[\exp\left(-2\epsilon^2 \cdot \tilde{Z}\right)\mb I(\tilde{Z}>0)\right]
&= 2 \cdot \sum_{j=1}^n {n\choose j} ( 1- p_h)^{n-j} p_h^j \exp\bigg(-2j\epsilon^2 \bigg)\\
&= 2 \cdot \sum_{j=1}^n {n\choose j} ( 1- p_h)^{n-j}\left( p_h\exp\bigg(-2\epsilon^2 \bigg)\right)^j\\     
&= 2 \cdot \left(1- p_h + p_h \cdot \exp\left(-2\epsilon^2  \right) \right)^n - 2 \cdot (1- p_h)^n ,
\end{aligned}
\end{equation}
where the last equality holds because of the binomial theorem. Hence, we have 
\begin{equation*}
\displaystyle
\begin{aligned}
\lvert\tilde{m}(\bs \gamma)- \hat{\mb{E}}_{\bm \gamma}[\ell(\bs x, \tilde{\bs\xi})] \rvert \leq\epsilon
\end{aligned}
\end{equation*}
with probability at least $1-(2\left(1- p_h + p_h \cdot \exp\left(-2\epsilon^2  \right) \right)^n - 2(1- p_h)^n)$.

Combining the results of (a) and (b), we have
\begin{align*}
  \mb P (\lvert \mb{E}_{\bm \gamma}[\ell(\bs x, \tilde{\bs\xi})] - \hat{\mb{E}}_{\bm \gamma}[\ell(\bs x, \tilde{\bs\xi})] \rvert >Lh + \epsilon)   &\leq 2\left(1- p_h + p_h \cdot \exp\left(-2\epsilon^2  \right) \right)^n - (1- p_h)^n\\
  &\leq 2 \left(1- p_h + p_h \cdot \exp\left(-2\epsilon^2 \right) \right)^n\\
  & \leq 2 \exp\left(-np_h  (1-\exp(-2\epsilon^2)\right)\\
  & \leq 2 \exp\left(-\frac{np_h \epsilon^2}{2}\right),
\end{align*}
where the second inequality comes from the fact that $(1+x)^n \leq \exp(nx)$, for any $x\in \mathbb [-1,\infty)$ and $n \in \mathbb N_+$ while the last inequality comes from the fact that $1-\exp(-2\epsilon^2) \geq \epsilon^2/2$ for any $\epsilon \in [0,1]$. Because the marginal density function satisfies $\underline{f} \leq f(\bm \gamma)$, we can bound $p_h$ by
\begin{equation*}
    c\underline{f}h^p \leq p_h
\end{equation*}
where $c$ is a constant that only depends on $p$; if $\| \cdot \|$ denotes the Euclidean norm, then $c=\frac{\pi^{p/2}}{\Gamma(p/2 + 1)}$ and $ch^p$ represents the $p$-dimensional volume of the norm ball $\{\bm \theta: \| \bm \theta \| \leq h\}$ \cite[Equation 5.19.4]{olver2016nist}. Thus, we have
\begin{equation*}
    \mb P (\lvert \mb{E}_{\bm \gamma}[\ell(\bs x, \tilde{\bs\xi})] - \hat{\mb{E}}_{\bm \gamma}[\ell(\bs x, \tilde{\bs\xi})] \rvert >Lh + \epsilon)   \leq 2 \exp\left(-\frac{n c\underline{f}h^p \epsilon^2}{2}\right).
\end{equation*}
This completes the proof. 
\end{proof} 

Equipped with Proposition~\ref{thm:navie_finite_bd}, we are now ready to prove Theorem~\ref{thm:suboptimality}.

\begin{proof}[Proof of Theorem \ref{thm:suboptimality}]
Under the boundedness assumption of $\mathcal X$, for any $\tau>0$, we can construct a finite subset $\mathcal X_\tau\subset X$ of cardinality $|\mathcal X_\tau|=\mathcal O(1)(D/\tau)^d$ such that for any $\bm x\in\mathcal X$, there exists $\bm x'\in\mathcal X_\tau$ with $\|\bm x-\bm x'\|\leq \tau$. 
From Proposition \ref{thm:navie_finite_bd}, we have for any fixed $\bm x'\in\mathcal X_\tau$, 
\begin{equation*}
  \mb P \left(\lvert \mb{E}_{\bm \gamma}[\ell(\bs x', \tilde{\bs\xi})] - \hat{\mb{E}}_{\bm \gamma}[\ell(\bs x', \tilde{\bs\xi})] \rvert >L_\gamma h + \epsilon\right) \leq   2 \exp\left(-\frac{nc\underline{f}h^p \epsilon^2}{2}\right).
\end{equation*} 
Applying union bound, we get with probability at least $1-2| X_\tau| \exp\left(-{nc\underline{f}h^p \epsilon^2}/{2}\right)$:
\begin{align*}
\;\lvert \mb{E}_{\bm \gamma}[\ell(\bs x', \tilde{\bs\xi})] - \hat{\mb{E}}_{\bm \gamma}[\ell(\bs x', \tilde{\bs\xi})] \rvert \leq L_\gamma h + \epsilon \quad\forall\bm x'\in\mathcal X_\tau.
\end{align*}
Next, by the $L_x$-Lipschitz continuity of the loss function, we have for any $\bm x\in\mathcal X$,
\begin{align*}
\lvert \mb{E}_{\bm \gamma}[\ell(\bs x, \tilde{\bs\xi})] - \hat{\mb{E}}_{\bm \gamma}[\ell(\bs x, \tilde{\bs\xi})] \rvert & \leq  \lvert \mb{E}_{\bm \gamma}[\ell(\bs x, \tilde{\bs\xi})] - {\mb{E}}_{\bm \gamma}[\ell(\bs x', \tilde{\bs\xi})] \rvert+\lvert \mb{E}_{\bm \gamma}[\ell(\bs x', \tilde{\bs\xi})] - \hat{\mb{E}}_{\bm \gamma}[\ell(\bs x', \tilde{\bs\xi})] \rvert\\
& \hspace{6cm}+ \lvert \hat{\mb{E}}_{\bm \gamma}[\ell(\bs x, \tilde{\bs\xi})] - \hat{\mb{E}}_{\bm \gamma}[\ell(\bs x', \tilde{\bs\xi})] \rvert\\
&\leq \lvert \mb{E}_{\bm \gamma}[\ell(\bs x', \tilde{\bs\xi})] - \hat{\mb{E}}_{\bm \gamma}[\ell(\bs x', \tilde{\bs\xi})] \rvert  + 2 L_x \tau
\end{align*}
for some $\bm x'\in\mathcal X_\tau$ with $\|\bm x-\bm x'\|\leq \tau$. Thus, $\lvert \mb{E}_{\bm \gamma}[\ell(\bs x, \tilde{\bs\xi})] - \hat{\mb{E}}_{\bm \gamma}[\ell(\bs x, \tilde{\bs\xi})] \rvert \leq L_\gamma h + \epsilon + 2 L_x \tau$ for all $\bm x\in\mathcal X$ with probability at least $1-2| X_\tau| \exp\left(-{nc\underline{f}h^p \epsilon^2}/{2}\right)$. 
Next, we set $\delta=2 |X_\tau|\exp\left(-nc\underline{f}h^p \epsilon^2/2\right)$ and solve for $\epsilon$ to obtain the generalization bound
\begin{equation*}
\lvert\mb{E}_{\bm \gamma}[\ell(\bs x, \tilde{\bs\xi})] - \hat{\mb{E}}_{\bm \gamma}[\ell(\bs x, \tilde{\bs\xi})] \rvert \leq L_\gamma h + \sqrt{\frac{2\log\left(\frac{2|X_\tau|}{\delta}\right)}{ nc\underline{f}h^p}} + 2 L_x \tau \qquad\forall\bm x\in\mathcal X
\end{equation*}
with probability at least $1-\delta$. 

Let $\hat {\bm x}$ be a minimizer of \eqref{eq:nadaraya-watson_problem} and $\bm x^\star$ be a minimizer of \eqref{eq:conditional_expectation_problem}. Then, we have 
\begin{align*}
\mb{E}_{\bm \gamma}[\ell(\hat{\bs x}, \tilde{\bs\xi})] - {\mb{E}}_{\bm \gamma}[\ell(\bs x^\star, \tilde{\bs\xi})] & = \mb{E}_{\bm \gamma}[\ell(\hat{\bs x}, \tilde{\bs\xi})] - \hat{\mb{E}}_{\bm \gamma}[\ell(\hat{\bs x}, \tilde{\bs\xi})]  + \hat{\mb{E}}_{\bm \gamma}[\ell(\hat{\bs x}, \tilde{\bs\xi})]  - {\mb{E}}_{\bm \gamma}[\ell(\bs x^\star, \tilde{\bs\xi})]\\
&\leq \mb{E}_{\bm \gamma}[\ell(\hat{\bs x}, \tilde{\bs\xi})] - \hat{\mb{E}}_{\bm \gamma}[\ell(\hat{\bs x}, \tilde{\bs\xi})]  + \hat{\mb{E}}_{\bm \gamma}[\ell({\bs x}^\star, \tilde{\bs\xi})]  - {\mb{E}}_{\bm \gamma}[\ell(\bs x^\star, \tilde{\bs\xi})]\\
&\leq 2L_\gamma h + 2\sqrt{\frac{2\log\left(\frac{2|X_\tau|}{\delta}\right)}{ nc\underline{f}h^p}}+4L_x\tau
\end{align*}
with probability at least $1-\delta$.
Here, the first inequality holds because $\bm x^\star$ is suboptimal to the Nadaraya-Watson approximation problem. 
To minimize the total error from the bias and the MAD terms, we minimize the right-hand side of the inequality and obtain the optimal bandwidth as 
\begin{equation}
    h^\star = \left( \frac{2L_\gamma ^2nc\underline{f}}{p^2\log\left(\frac{2|X_\tau|}{\delta}\right)}\right)^{\frac{-1}{p+2}}.
\end{equation}
Substituting the bandwidth with the above
\begin{align*}
\mb{E}_{\bm \gamma}[\ell(\hat{\bs x}, \tilde{\bs\xi})] - {\mb{E}}_{\bm \gamma}[\ell(\bs x^\star, \tilde{\bs\xi})] \leq 2L_\gamma^{\frac{p}{p+2}} \frac{p+2}{(4p^p)^{\frac{1}{p+2}}}  \left(\frac{2\log\left(\frac{2|X_\tau|}{\delta}\right)}{nc\underline{f}} \right)^{\frac{1}{p+2}}+4L_x\tau
\end{align*}
with probability at least $1-\delta$.
Setting $|X_\tau|=\mathcal O(1)(D/\tau)^d$  yields the desired result. 
\end{proof}

\section{Concluding Remarks}
We have presented the first finite-sample generalization bounds for contextual stochastic optimization via kernel regression. The derivation is based on a mean-MAD (mean absolute deviation) decomposition of the error, diverging from the traditional mean-variance decomposition commonly used in the statistics literature. Our findings offer practical guidance on selecting the optimal bandwidth parameter and determining the sufficient number of samples to achieve a desired suboptimality gap. Furthermore, our proof technique may find relevance in other contextual stochastic optimization settings that utilize a biased estimator of risk measures. 

We remark that our result is conservative since we aimed for a simple analytical form comparable to existing bounds for sample-average approximation in stochastic programming. Improved results could be obtained by considering any valid intermediate bound in the proof.
As expected, due to the minimal assumptions made about the joint distribution, the sample complexity suffers from the curse of dimensionality in the covariates. Future research will focus on developing dimensionality reduction techniques and identifying specific distributional settings that enable improved sample complexity. Additionally, it will be interesting to extend the results to the multi-stage setting to address data-driven dynamic stochastic optimization problems.

\paragraph{Acknowledgements} 
Grani A.~Hanasusanto was supported by the National Science Foundation under grants 2342505 and 2343869. Chin Pang Ho was supported by the Research Grants Council (RGC) under the General Research Fund 11508623 and the CityU Start-Up Grant 9610481.

\bibliographystyle{abbrv}
\bibliography{references}

\newpage



\end{document}